\documentclass[12pt,reqno]{amsart}

\usepackage{epsfig}
\usepackage{amsmath}
\usepackage{amsthm}
\usepackage{amsfonts}
\usepackage{graphicx}

\numberwithin{equation}{section}

\newcommand{\lltwos}{\ell^2({\mathbb S})}

\newcommand{\ssup}[2]{\underset{#1\in #2}{\textrm{sup }}}

\newcommand{\D}{{\mathbb D}}
\newcommand{\A}{{\mathbb A}}

\newcommand{\C}{{\mathbb C}}
\newcommand{\N}{{\mathbb N}}

\newcommand{\Z}{{\mathbb Z}}

\renewcommand{\d}{\partial}

\newcommand{\f}{\varphi}

\newtheorem{theo}{{\sc \bf Theorem}}[section]

\newtheorem{lem}[theo]{{\sc \bf Lemma}}
\newtheorem{prop}[theo]{{\sc \bf Proposition}}
\newtheorem{defin}[theo]{{\sc \bf Definition}}

\begin{document}

\title{Classical limit of the d-bar operators on quantum domains}

\author{Slawomir Klimek}
\address{Department of Mathematical Sciences,
Indiana University-Purdue University Indianapolis,
402 N. Blackford St., Indianapolis, IN 46202, U.S.A.}
\email{sklimek@math.iupui.edu}

\author{Matt McBride}
\address{Department of Mathematical Sciences,
Indiana University-Purdue University Indianapolis,
402 N. Blackford St., Indianapolis, IN 46202, U.S.A.}
\email{mmcbride@math.iupui.edu}

\thanks{}

\date{\today}

\begin{abstract}
We study one parameter families $D_t$, $0<t<1$ of non-commutative analogs of the d-bar operator $D_0 = \frac{\d}{\d\overline{z}}$ on disks and annuli in complex plane and show that, under suitable conditions, they converge in the classical limit to their commutative counterpart. More precisely, we endow the corresponding families of Hilbert spaces with the structures of continuous fields over the interval $[0,1)$ and we show that the inverses of the operators $D_t$ subject to APS boundary conditions form morphisms of those continuous fields of Hilbert spaces.
\end{abstract}

\maketitle

%\noindent {\bf Keywords} Quantum spaces, Noncommutative geometry, Dirac operators, APS boundary conditions
%
%\noindent {\bf Mathematics Subject Classification (2010)} 46L87 

\section{Introduction}

According to the broadest and the most flexible definition, a quantum space is simply a noncommutative algebra. Noncommutative geometry \cite{Connes} studies what could be considered ``geometric properties" of such quantum spaces.

One of the most basic examples of a quantum space is the quantum unit disk  $C(\D_t)$ of \cite{KL}.
It is defined as the universal unital $C^*$-algebra with the generators $z_t$ and $\overline{z_t}$ which are adjoint to each other, and satisfy the following commutation relation: $[\overline{z_t},z_t] = t(I-z_t\overline{z_t})(I-\overline{z_t}z_t)$, for a continuous parameter $0<t<1$. 

It was proved in \cite{KL} that $C(\D_t)$ has a more concrete representation as the $C^*$-algebra generated by the unilateral weighted shift with
the weights given by the formula:
\begin{equation}\label{weightexample}
w_t(k) = \sqrt{\frac{(k+1)t}{1+(k+1)t}}.
\end{equation}
In fact, as a $C^*$-algebra, $C(\D_t)$ is isomorphic to the Toeplitz algebra. Moreover the family $C(\D_t)$ is a deformation, and even deformation - quantization of the algebra of continuous functions on the disk $C(\D)$ obtained in the limit as $t\to 0$, called the classical limit.

The quantum unit disk is one of the simplest examples of a quantum manifold with boundary. It is also an example of a quantum complex domain, with $z_t$ playing the role of a quantum complex coordinate. Additionally, biholomorphisms of the unit disk naturally lift to automorphisms of $C(\D_t)$, see \cite{KL}.

In view of this complex analytic interpretation of the quantum unit disk, there is a natural need to define analogs of complex partial derivatives as some kind of unbounded operators on $C(\D_t)$ and its various Hilbert space completions. Such constructions have been described in several places in the literature, see for example \cite{BKLR}, \cite{CKW}, \cite{K}, \cite{KM}, \cite{KM1},\cite{V}. In this paper we are primarily concerned with one such choice, the so-called balanced d and d-bar operators of \cite{KM} which we describe below.

One notices that $S_t:=[\overline{z_t},z_t]$ is an invertible trace class operator (with an unbounded inverse) and defines

\begin{equation*}
D_ta = S_t^{-1/2}[a,z_t]S_t^{-1/2} 
\end{equation*}
and
\begin{equation*}
\overline{D_t}a = S_t^{-1/2}[\overline{z_t},a]S_t^{-1/2},
\end{equation*}
for appropriate $a\in C(\D_t)$.
These two operators have the following easily seen properties

\begin{equation*}
\begin{aligned}
&D_t(1) = 0, \quad D_t(z_t) = 0, \quad D_t(\overline{z_t}) = 1 \\
&\overline{D_t}(1) = 0,\quad \overline{D_t}(z_t) = 1,\quad \overline{D_t}(\overline{z_t}) = 0
\end{aligned}
\end{equation*}
which makes them plausible candidates for quantum complex partial derivatives. To make an even better case of their suitability, one would like to know that in some kind of interpretation of the limit as $t\to 0$, they indeed become the classical partial derivatives. This problem was posed at the end of \cite{K} and it is the subject of the present paper. 

In fact we consider here a broader classical limit problem by studying quite general families of unilateral weights $w_t(k)$, and not just those given by (\ref{weightexample}). Like in \cite{KM} such unilateral shifts are still considered coordinates of quantum disks. Additionally we also consider bilateral shifts and the $C^*$-algebras they generate. They are quantum analogs of annuli and can be analyzed very similarly to the quantum disks. 

We start with giving a concrete meaning to the classical limit $t\to 0$, which involves two important steps. The first step is to consider certain bounded functions of the quantum d and d-bar operators to properly manage their unboundedness. In this paper we choose to work with the inverses of the operators $D_t$ subject to APS boundary conditions \cite{APS} since they are easy to describe and
we can use the results of \cite{CKW}, \cite{KM}.

The second step of our approach to the classical limit is the choice of framework for studying limits of objects living in different spaces. Such a natural framework is provided by the language of continuous fields, in our case of continuous fields of Hilbert spaces, see \cite{Dixmier}. Following \cite{CKW} and \cite{KM} we define, using operators $S_t$, weighted Hilbert space completions $\mathcal{H}_t$, $0<t<1$, of the above quantum domains, while $\mathcal{H}_0$ is the classical $L^2$ space. We then equip that family of Hilbert spaces with a natural structure of continuous field, namely the structure generated by the polynomials in complex quantum and classical coordinates.
In this setup the study of the classical limit becomes a question of continuity, a property embedded in the definition of the continuous field. Consequently, inverses of the operators $D_t$ subject to APS boundary conditions, are considered as morphisms of the continuous fields of Hilbert spaces.
The main result of this paper is that in such a sense the limit of $\overline{D_t}$ is indeed 
$\frac{\d}{\d\overline{z}}$. 

The paper is organized as follows.   In section 2 we review the definitions and properties of continuous fields of Hilbert spaces and their morphisms.  In Section 3 we describe the constructions of the quantum disk, the quantum annulus, Hilbert spaces of $L^2$ ``functions" on those quantum spaces, d-bar operators and their inverses subject to APS conditions. We state the conditions on weights $w_t(k)$ and provide example of such weights. We construct the generating subspace $\Lambda$ needed for the construction of the continuous field of Hilbert spaces. The main results of this paper are also formulated at the end of that section.  Finally, section 4 contains the proofs of the results.

\section{Continuous Fields of Hilbert Spaces}
In this section we review some aspects of the theory of continuous fields of Hilbert spaces. The main reference here is Dixmier's book \cite{Dixmier}.

\begin{defin}
A {\it continuous field of Hilbert spaces} is a triple, denoted $(\Omega, \mathcal{H}, \Gamma)$, where $\Omega$ is a locally compact topological space, $\mathcal{H} = \{H(\omega) \ : \ \omega\in\Omega\}$ is a family of Hilbert spaces, and $\Gamma$ is a linear subspace of $\prod_{\omega\in\Omega} H(\omega)$, such that the following conditions hold

\begin{enumerate}
  \item for every $\omega\in\Omega$, the set of $x(\omega)$, $x\in \Gamma$, is dense in $H(\omega)$,
  \item for every $x\in \Gamma$, the function $\omega\mapsto\|x(\omega)\|$ is continuous,
  \item let $x\in\prod_{\omega\in\Omega}H(\omega)$; if for every $\omega_0\in\Omega$ and every $\varepsilon>0$, there exists  $x'\in\Gamma$ such that $\|x(\omega)-x'(\omega)\|\leq \varepsilon$ for every $\omega$ in some neighborhood (depending on $\varepsilon$) of $\omega_0$, then $x\in\Gamma$. 
\end{enumerate}
\end{defin}

The point of this definition is to describe a continuous arrangement of a family of different Hilbert spaces. If they are all the same, then the space $\Gamma$ of continuous functions on $\Omega$ with values in that Hilbert space clearly satisfies all the conditions.

Below we will use the following terminology. We say that a section $x\in\prod_{\omega\in\Omega}H(\omega)$ is approximable by $\Gamma$ at $\omega_0$ if for every $\varepsilon>0$, there exists an $x'\in\Gamma$ and a neighborhood of $\omega_0$ such that $\|x(\omega)-x'(\omega)\|\leq \varepsilon$ for every $\omega$ in that neighborhood. In this terminology condition 3 of the above definition says that if a section $x$ is approximable by $\Gamma$ at every $\omega\in\Omega$, then $x\in\Gamma$.

The above definition is a little cumbersome to work with, namely, trying to describe $\Gamma$ in full detail is usually very difficult since the third condition isn't easy to verify. The following proposition, proved in \cite{Dixmier}, makes it easier to construct continuous fields.

\begin{prop}\label{cont_field_def}
Let $\Omega$ be a locally compact topological space, and let $\mathcal{H} = \{H(\omega) \ : \ \omega\in\Omega\}$ be a family of Hilbert spaces.   If $\Lambda$ is a linear subspace  of $\prod_{\omega\in\Omega} H(\omega)$ such that

\begin{enumerate}
  \item for every $\omega\in\Omega$, the set of $x(\omega)$, $x\in \Lambda$, is dense in $H(\omega)$,
  \item for every $x\in\Lambda$, the function $\omega\mapsto\|x(\omega)\|$ is continuous,
\end{enumerate}
then $\Lambda$ extends uniquely to a linear subspace $\Gamma \subset \prod_{\omega\in\Omega} H(w)$ such that $(\Omega, \mathcal{H}, \Gamma)$ is a continuous field of Hilbert spaces.
\end{prop}

Here one says that if a linear subspace  $\Lambda$ of $\prod_{\omega\in\Omega} H(\omega)$ satisfies the two conditions above then $\Lambda$ generates the continuous field of Hilbert spaces $(\Omega, \mathcal{H}, \Gamma)$.  In fact, $\Gamma$ is simply constructed as a local completion of $\Lambda$, i.e. $\Gamma$ consists of all those sections $x\in\prod_{\omega\in\Omega}H(\omega)$ which are approximable by $\Lambda$ at every $\omega\in\Omega$. 

Next we consider morphisms of continuous fields of Hilbert spaces. For this we have the following definition.

\begin{defin}
Let $(\Omega, \mathcal{H}, \Gamma)$ be a continuous field of Hilbert spaces and let $T(\omega) : H(\omega)\to H(\omega)$ be a collection of operators acting on the Hilbert spaces $H(\omega)$.   Define $T = \prod_{\omega\in\Omega} T(\omega) : \prod_{\omega\in\Omega} H(\omega)\to \prod_{\omega\in\Omega} H(\omega)$.   We say that $\{T(\omega)\}$ is a continuous family of bounded operators in $(\Omega, \mathcal{H}, \Gamma)$ if 

\begin{enumerate}
  \item $T(\omega)$ is bounded for each $\omega$,
  \item $\ssup{\omega}{\Omega}\|T(\omega)\|<\infty$,
  \item $T$ maps $\Gamma$ into $\Gamma$.
\end{enumerate}  

\end{defin}

The proposition below contains an alternative description of the third condition above, so it is more manageable.

\begin{prop}\label{morphprop}
With the notation of the above definition, the following three conditions are equivalent:
\begin{enumerate}
  \item $T$ maps $\Gamma$ into $\Gamma$,
  \item $T$ maps $\Lambda$ into $\Gamma$,
  \item for every $x\in \Lambda$ and for every $\omega\in\Omega$, $T(\omega)x(\omega)$ is approximable by $\Lambda$ at $\omega$.
\end{enumerate}  
\end{prop}

\begin{proof}
The items above are arranged from stronger to weaker.
The proof that $(2)$ is equivalent to $(3)$ is a simple consequence of the way that $\Gamma$ is obtained from $\Lambda$ described in the paragraph following Proposition \ref{cont_field_def}. Condition $(2)$ implies condition $(1)$ because
$\ssup{\omega}{\Omega}\|T(\omega)\|<\infty$ and so, if $x(\omega)$ and $y(\omega)$ are locally close to each other, so are $T(\omega)x(\omega)$ and $T(\omega)y(\omega)$.
\end{proof}

\section{D-Bar operators on non-commutative domains}
In this section we review a variety of constructions needed to formulate and prove the results of this paper. Those constructions include the definitions of the quantum disk, the quantum annulus, Hilbert spaces of $L^2$ ``functions" on those quantum spaces, and d-bar operators that were discussed in \cite{KM}. Other items discussed in this section are APS boundary conditions, inverses of d-bar operators subject to APS conditions, conditions on weights, and a construction of the generating subspace $\Lambda$ of the continuous field of Hilbert spaces. The  
main results are stated at the end of this section. 

In the following formulas we let $\mathbb{S}$ be either $\N$ or $\Z$. Let $t\in (0,1)$ be a parameter. 
Let $\{e_k\}$, $k\in \mathbb{S}$ be the canonical basis for $\lltwos$. Given a $t$-dependent,
bounded sequence of numbers $\{w_t(k)\}$, called weights, the weighted shift $U_{w_t}$ is an operator in $\lltwos$ defined by:
$U_{w_t} e_k = w_t(k)e_{k+1}$. The usual shift operator $U$ satisfies $Ue_k = e_{k+1}$.

%We will need the diagonal operator $W_t$ given by $W_t e_k = w_t(k)e_k$,
%and we note that  $U_{w_t}=U{W_t}$ and $W_t=(U_{w_t}^*U_{w_t})^{1/2}$.  

If $\mathbb{S}=\N$ then the shift $U_{w_t}$ is called a unilateral shift and it will be used to define a quantum disk. If $\mathbb{S}=\Z$ then the shift $U_{w_t}$ is called a bilateral shift and we will use it to define a quantum annulus (also called a quantum cylinder). For the choice of weights \ref{weightexample} the shifts $U_{w_t}$ are the quantum complex coordinates $z_t$ of the introduction.

We require the following condition on the one-parameter family of weights $w_t(k)$.

\medskip

{\it Condition $1$.\ } The weights $w_t(k)$ form a positive, bounded, strictly increasing sequence in $k$ such that the limits 
$w_\pm:=\lim\limits_{k\to\pm\infty}w_t(k)$ exist, are positive, and independent of $t$.
\medskip

Consider the commutator $S_t= U_{w_t}^*U_{w_t} - U_{w_t}U_{w_t}^*$. It is a diagonal operator
$S_te_k= S_t(k)\,e_k$, where
\begin{equation*}
S_t(k):=w_t(k)^2 - w_t(k-1)^2.
\end{equation*}
Moreover $S_t$ is a trace class operator with easily computable trace: 
\begin{equation}\label{traceformula}
\textrm{tr}(S_t)=\sum_{k\in\mathbb{S}}S_t(k)=(w_+)^2-(w_-)^2
\end{equation}
in the bilateral case, and $\textrm{tr}(S_t)=(w_+)^2$ in the unilateral case. Additionally $S_t$ is invertible with unbounded inverse.

We assume further conditions on the $w_t(k)$'s and the $S_t(k)$'s. Those conditions were simply extracted from
the proofs in the next section to make the estimates work. They are possibly not optimal, but they cover our motivating example described in the introduction.

\medskip

{\it Condition $2$.\ } The function $t\mapsto w_t(k)$ is continuous for every $k$, and for every $\varepsilon>0$, $w_t(k)$ converges to $w_\pm$ as $k\to\pm\infty$ uniformly on the interval $t\geq \varepsilon$.

\medskip

{\it Condition $3$.\ } If $h_1(t):=\ssup{k}{\mathbb{S}}S_t(k)$ then $h_1(t)\to 0$ as $t\to 0^+$.

\medskip

{\it Condition $4$.\ } The supremum $h_2(t):= \ssup{k}{\mathbb{S}}\left|1-\frac{S_t(k+1)}{S_t(k)}\right|$ exists, and is a bounded function of $t$, and $h_2(t)\to 0$ as $t\to0^+$.

\medskip

{\it Condition $5$.\ } The supremum $h_3(k):=\ssup{t}{[0,1)}\left|1-\frac{w_t(k-1)}{w_t(k)}\right|$ exists for every $k$, and $h_3(k)\to 0$ as $k\to\pm\infty$.

\medskip
Notice that the last condition implies that
\begin{equation}\label{wconstineq}
w_t(k)\leq \textrm{const}\,w_t(k-1)
\end{equation}
where the const above does not depend on $t$ and $k$. This observation will be used in the proofs in the next section.

Before moving on, we verify that the weight sequence \ref{weightexample} in the example in the introduction satisfies all of the conditions.   First we compute:
\begin{equation*}
S_t(k) = \frac{t}{(1+kt)(1+(k+1)t)}.
\end{equation*}
{\it Conditions} 1 and 2 are all easily seen to be true with $w_+=1$.   For {\it conditions} 3, 4, and 5 simple computations give $h_1(t) = t/(1+t)$, $h_2(t) = 2t/(1+2t)=O(t)$, and $h_3(k) = (k+1+\sqrt{k^2+k})^{-1}=O(1/k)$, and so, by inspection,  these weights meet all the required conditions. 
Examples of bilateral shifts satisfying the above conditions are:
\begin{equation*}
w_t^2(k)=\alpha+\beta\frac{tk}{1+t|k|}.
\end{equation*}
For this example $h_1(t)=\beta t/(1+t)$ and $h_2(t) = O(t)$, $h_3(k) = O(1/k)$, $w_+^2=\alpha +\beta$, $w_-^2=\alpha -\beta$. Another similar example is
$w_t^2(k)=\alpha+\beta\tan^{-1}(tk)$.

Next we proceed to the definition of the continuous field of Hilbert spaces over the interval $I=[0,1)$.
Let $C^*(U_{w_t})$ be the $C^*$-algebra generated by $U_{w_t}$.  Then, in the unilateral case, the algebra $C^*(U_{w_t})$ is called the non-commutative disk. There is a canonical map:

\begin{equation*}
C^*(U_{w_t}) \overset{r}{\longrightarrow} C(S^1)
\end{equation*}
called the restriction to the boundary map.

In the bilateral case the algebra $C^*(U_{w_t})$  is called the non-commutative cylinder, and we also have restriction to the boundary maps:
\begin{equation*}
C^*(U_{w_t}) \overset{r=r_+\oplus r_-}{\longrightarrow} C(S^1) \oplus C(S^1).
\end{equation*}

We then define the Hilbert space $\mathcal{H}_t$, for $t>0$, to be the completion of $C^*(U_{w_t})$ with respect to the inner product given by:

\begin{equation*}
\|a\|_t^2=\textrm{tr}\left(S_t^{1/2}aS_t^{1/2}a^*\right).
\end{equation*}
For $t=0$ we set $\mathcal{H}_0 = L^2(\D_{w_+})$ in the unilateral/disk case and $\mathcal{H}_0 = L^2(\A_{w_-,w_+})$ in the bilateral/annulus case where $\D_{w_+}:=\{z\in\C:|z|\leq w_+\}$ is the disk of radius $w_+$, and $\A_{w_-,w_+}:=\{z\in\C:w_-\leq|z|\leq w_+\}$ is the annulus with inner radius $w_-$ and outer radius $w_+$. In what follows we usually skip the norm subscript as it will be clear from other terms subscript which Hilbert space norm or operator norm is used. Also notice that setting $w_-=0$ reduces most annulus formulas below to the disk case.

So far we have the space of parameters $I$, and for every $t\in I$ we defined the Hilbert space $\mathcal{H}_t$. We also have distinguished 
elements of $\mathcal{H}_t$, namely quantum complex coordinates $U_{w_t}$. We use them to generate the continuous field of Hilbert spaces. More precisely we 
define the generating linear space $\Lambda \subset \prod_{t\in I}\mathcal{H}_t$ to consists of all those
$x=\{x(t)\}$ such that there exists $N>0$, (depending on $x$), and for every $n\leq N$ there are functions $f_n,g_n\in C([{(w_-)^2},{(w_+)^2}])$, such that for $t>0$:

\begin{equation}\label{lambdadef}
x_t(k) = \sum_{n\le N}U^nf_n\left(w_t(k)^2\right) + \sum_{n\le N}g_n\left(w_t(k)^2\right)\left(U^*\right)^n,
\end{equation}
and for $t=0$: 
\begin{equation}\label{lambdadefzero}
x_0(r,\f) = \sum_{n\le N}f_n(r^2)e^{in\f} + \sum_{n\le N}g_n(r^2)e^{-in\f}.
\end{equation}

Now we proceed to the definitions of the quantum  d-bar operators. 
The operator $D_t$ in $\mathcal{H}_t$ is given by the following expression:

\begin{equation*}
D_ta = S_t^{-1/2}\left[a,U_{w_t}\right]S_t^{-1/2}
\end{equation*}
for $t>0$, and for $t=0$, $D_0 = \d/\d\overline{z}$. Of course we need to specify the domain of $D_t$ since  it is an unbounded operator.  For reasons indicated in the introduction, in this paper we consider the operators subject to the APS boundary conditions.   Let $P^{\pm}$ be the spectral projections in $L^2(S^1)$ of the boundary operators $\pm \frac{1}{i}\frac{\d}{\d\theta}$ onto the interval $(-\infty,0]$.
The domain of $D_t$  is then defined to be:

\begin{equation*}
\textrm{dom}(D_t) = \{a\in \mathcal{H}_t\ : \ \|D_ta\| <\infty , \ r(a)\in \textrm{Ran }P^+\}
\end{equation*}
for the disk. For the annulus we set:

\begin{equation*}
\textrm{dom}(D_t) = \{a\in \mathcal{H}_t\ : \ \|D_ta\| < \infty , \ r_+(a)\in \textrm{Ran }P^+,\ r_-(a)\in\textrm{Ran }P^-\}.
\end{equation*}
Here the maps $r$, $r_\pm$ are the restriction to the boundary maps, that by the results of \cite{KM}, continue to make sense for those $a\in\mathcal{H}_t$
for which $\|D_ta\| < \infty$.

If $t=0$ the domain of $D_0$ consists of all those first Sobolev class functions $f$ on the disk or the annulus for which the APS condition holds i.e. either
$r(f)\in \textrm{Ran }P^+$ or $r_+(f)\in \textrm{Ran }P^+,\ r_-(f)\in\textrm{Ran }P^-$, depending on the case.
Here,  by slight notational abuse, the symbols $r$, $r_\pm$ are the classical restriction to the boundary maps.

It was verified in \cite{KM} that the above defined operators $D_t$ are invertible, with bounded, and even compact inverses $Q_t$.
Using \cite{KM} we can write down the formulas for $Q_t$. If $x\in\Lambda$ we have the following for $t>0$:

\begin{equation*}
\begin{aligned}
&Q_tx_t(k) = \\
&-\sum_{n=0}^N U^{n}\left(\sum_{i\geq k}\frac{w_t(k+1)\cdots w_t(k+n)}{w_t(i+1)\cdots w_t(i+n)}\cdot\frac{S_t(i)^{1/2}S_t(i+n+1)^{1/2}}{w_t(k+n)}f_{n+1}(w_t(i)^2)\right) \\
&+ \sum_{n=1}^N \left(\sum_{i\leq k}\frac{w_t(i)\cdots w_t(i+n-1)}{w_t(k)\cdots w_t(k+n-1)}\cdot\frac{S_t(i)^{1/2}S_t(i+n-1)^{1/2}}{w_t(i+n-1)}g_{n-1}(w_t(i)^2)\right)\left(U^*\right)^{n}.
\end{aligned}
\end{equation*}
For the disk the second sum is from $0$ to $k$, while for the annulus it is from $-\infty$ to $k$.   

For $t=0$ we have

\begin{equation*}
D_0x_0 = \sum_{n=0}^N \frac{e^{i(n+1)\theta}}{2}\left(2rf_n'(r^2)-\frac{n}{r}f_n(r^2)\right) + \sum_{n=1}^N\left(2rg_n'(r^2) +\frac{n}{r}g_n(r^2)\right)\frac{e^{-i(n-1)\theta}}{2}.
\end{equation*}
for both the disk and annulus.  From this we can compute the inverse $Q_0$ of $D_0$. A straightforward calculation
gives the following result:

\begin{equation*}
Q_0x_0 = -\sum_{n=0}^N {e^{in\theta}}\int_{r^2}^{(w_+)^2}f_{n+1}(\rho^2)\frac{r^{n-1}}{\rho^n}d(\rho^2) + \sum_{n=1}^N {e^{-in\theta}}\int_{(w_-)^2}^{r^2}g_{n-1}(\rho^2)\frac{\rho^{n-1}}{r^n}d(\rho^2),
\end{equation*}
for the annulus, and the same formula with $w_-$ replaced by $0$ for the disk.   

We are now in the position to state the main results of the paper.   They are summarized in the following two theorems:

\begin{theo}\label{cont_hil_sp}
Given $I=[0,1)$, let $\mathcal{H}= \{\mathcal{H}_t\ : \ t\in I\}$ be the family of Hilbert spaces defined above and let $\Lambda$ be the linear subspace of 
$\prod_{t\in I}\mathcal{H}_t$ defined by \ref{lambdadef} and \ref{lambdadefzero}.  Also let the conditions on $w_t(k)$ and $S_t(k)$ hold.   Then $\Lambda$ generates a continuous field of Hilbert spaces denoted below by $(I, \mathcal{H}, \Gamma)$.
\end{theo}

\begin{theo}\label{cont_fam_oper}
Let $Q_t : \mathcal{H}_t\to \mathcal{H}_t$ be the collection of operators for $t\in[0,1)$ defined above.  Then $\{Q_t\}$ is a continuous family of bounded operators in the continuous field $(I, \mathcal{H}, \Gamma)$.
\end{theo}

We finish this section by shortly indicating that the above results are also valid for families of d-bar operators studied in \cite{CKW}.
Let us quickly review the differences.  The Hilbert space $\mathcal{H}_t$ studied in \cite{CKW} is the completion of $C^*(U_{w_t})$ with respect to the following inner product:

\begin{equation*}
\|a\|_t^2 = \textrm{tr}(S_taa^*).
\end{equation*}
The quantum d-bar operator $D_t$ of \cite{CKW}, acting in $\mathcal{H}_t$, is given by the following formula:

\begin{equation*}
D_t a = S_t^{-1}[a,U_{w_t}].
\end{equation*}
It turns out that Theorems \ref{cont_hil_sp} and \ref{cont_fam_oper} are also true for the above spaces and operators. In fact the proofs are even easier in this case and {\it Condition} 4, designed to handle expressions like $S_t(k+n)^{1/2}S_t(k)^{1/2}$ is
not even needed.

The next section will contain all the analysis needed to prove the two theorems. 

\section{Continuity and the classical limit}
We will prove the  two theorems  from the above section by a series of steps that verify the assumptions in the definitions of the continuous field of Hilbert spaces and the continuous family of bounded operators. We concentrate mainly on the annulus case. The disk case is in some respects simpler. Most of the formulas for the annulus are true also in the disk case with a modification: replacing $w_-$ by zero. The summation index in the annulus case extends to $-\infty$ and in couple of places the corresponding sums need to be estimated. This is not the issue in the disk case where the summation starts at zero. However the major difficulty in the disk case are the $w_t$ terms in the denominator in the formula for the parametrix, since they go to zero as $t$ goes to zero. The proofs we describe below work in both cases, but much shorter arguments are possible in the annulus case.

We first verify that $\Lambda$ generates a continuous field of Hilbert spaces. To this end we need to check two things:  the density in $\mathcal{H}_t$ of $x(t)$, $x\in\Lambda$, and the continuity of the norm. The density is immediate, since, for example, the canonical basis elements of $\mathcal{H}_t$, see the proof of Lemma 5.1 in \cite{KM}, come from $\Lambda$.

The verification of the continuity of the norm is done in two steps: continuity at $t=0$, and at $t>0$. If $x\in\Lambda$, i.e. $x$ is given by formulas \ref{lambdadef} and \ref{lambdadefzero}, then the norm of $x_t$ in $\mathcal{H}_t$ is, for $t>0$, given by

\begin{equation}\label{xnormformula}
\begin{aligned}
\|x_t\|^2 &= \sum_{n=0}^N\sum_{k\in\mathbb{S}} S_t(k+n)^{1/2}S_t(k)^{1/2}\left|f_n\left(w_t(k)^2\right)\right|^2 \\
&+ \sum_{n=1}^N\sum_{k\in\mathbb{S}} S_t(k+n)^{1/2}S_t(k)^{1/2}\left|g_n\left(w_t(k)^2\right)\right|^2 ,
\end{aligned}
\end{equation}
while for $t=0$ the norm of $x_0$ is

\begin{equation}\label{xzeronormformula}
\|x_0\|^2 = \sum_{n=0}^N \int_{(w_-)^2}^{(w_+)^2}\left|f_n\left(r^2\right)\right|^2d\left(r^2\right) + \sum_{n=1}^N\int_{(w_-)^2}^{(w_+)^2}\left|g_n\left(r^2\right)\right|^2d\left(r^2\right) .
\end{equation}

The next lemma is needed to handle the product of $S$ terms with different arguments.
\begin{lem}\label{stermslemma}
For $n\geq 1$ we have
\begin{equation*}
\ssup{k}{\mathbb{S}}\left|\frac{S_t(k+n)}{S_t(k)}-1\right|\leq (2+h_2(t))^{n-1}h_2(t), 
\end{equation*}
where $h_2(t)$ is the function defined in {\it Condition} 4.
\end{lem}
\begin{proof} The proof is by induction. For $n=1$ we get {\it Condition} 4. The inductive step is
\begin{equation*}
\begin{aligned}
\left|\frac{S_t(k+n+1)}{S_t(k)}-1\right|&= \left|\frac{S_t(k+n+1)}{S_t(k+n)}\left(\frac{S_t(k+n)}{S_t(k)}-1\right)+\frac{S_t(k+n+1)}{S_t(k+n)}-1\right|\leq\\
&\leq(1+h_2(t))\,(2+h_2(t))^{n-1}h_2(t)+h_2(t)\leq (2+h_2(t))^{n}h_2(t)\\
\end{aligned}
\end{equation*}
and the lemma is proved.
\end{proof}

Now we are ready to discuss the continuity of norms \ref{xnormformula}, \ref{xzeronormformula} as $t\to 0^+$. 

\begin{prop}\label{normconvergenceuni}
If $x_t$ is in $\Lambda$ then

\begin{equation*}
\lim_{t\to 0^+} \|x_t\| = \|x_0\|
\end{equation*}
\end{prop}

\begin{proof}
Without loss of generality we may assume that $x_t(k)=U^nf_n\left(w_t(k)^2\right)$ and $x_0(r,\f) = f_n(r^2)e^{in\f}$
, as the proof is identical for the $g$ terms, and the elements of $\Lambda$ are finite sums of such $x$'s. We have
\begin{equation*}
\begin{aligned}
\left|\|x_t\|^2-\|x_0\|^2\right| &=\left|\sum_{k\in\mathbb{S}} S_t(k+n)^{1/2}S_t(k)^{1/2}\left|f_n(w_t(k)^2)\right|^2-\int_{(w_-)^2}^{(w_+)^2}\left|f_n(r^2)\right|^2d(r^2)\right| \\
&\le\left|\sum_{k\in\mathbb{S}} S_t(k)\left|f_n(w_t(k)^2)\right|^2-\int_{(w_-)^2}^{(w_+)^2}\left|f_n(r^2)\right|^2d(r^2)\right|+ \\
&+ \left|\sum_{k\in\mathbb{S}} \left(S_t(k+n)^{1/2}S_t(k)^{1/2}-S_t(k)\right)\left|f_n(w_t(k)^2)\right|^2\right|.
\end{aligned}
\end{equation*}
Since $f_n$ is continuous and hence bounded, we can estimate:

\begin{equation*}
\begin{aligned}
&\left|\|x_t\|^2-\|x_0\|^2\right| \le \left|\sum_{k\in\mathbb{S}} S_t(k)\left|f_n(w_t(k)^2)\right|^2-\int_{(w_-)^2}^{(w_+)^2}\left|f_n(r^2)\right|^2d(r^2)\right|+ \\
&+\textrm{const}\left|\sum_{k\in\mathbb{S}} S_t(k)\left[\left(\frac{S_t(k+n)}{S_t(k)}\right)^{1/2}-1\right]\right|.
\end{aligned}
\end{equation*}
Using $S_t(k)=w_t(k)^2 - w_t(k-1)^2$ and {\it Condition $3$}, we see that the first term inside of the absolute value is a difference of a Riemann sum and the integral to which it converges as $t\to 0^+$. Hence this term is zero in the limit. As for the second term, since by \ref{traceformula}, $\sum_{k\in\mathbb{S}} S_t(k)=(w_+)^2-{(w_-)^2}=\textrm{const}$, Lemma \ref{stermslemma} shows that it also goes to zero, because, by {\it Condition} 4, $h_2(t)\to 0$ as $t\to 0^+$.
\end{proof}

We can now prove the first theorem.

\begin{proof}(of Theorem \ref{cont_hil_sp})
We have already verified that $\Lambda$  satisfies some of the properties of Proposition \ref{cont_field_def}.   
What remains is the proof of the continuity of the norm for $t>0$. Notice that by {\it Condition} 2 all the terms in formula \ref{xnormformula} are continuous in $t$, $t>0$. Thus we need to show that the series \ref{xnormformula} converges  uniformly in $t$ (away from $t=0$). Assuming again that $x_t(k)=U^nf_n\left(w_t(k)^2\right)$, and using the boundedness of $f_n$ we have:
\begin{equation*}
\begin{aligned}
&\left|\|x_t\|^2- \sum_{k=L+1}^{M-1} S_t(k+n)^{1/2}S_t(k)^{1/2}\left|f_n(w_t(k)^2)\right|^2\right|\leq \\
&\leq\textrm{const}\,\sum_{k\geq M} S_t(k+n)^{1/2}S_t(k)^{1/2}+\textrm{const}\,\sum_{k\leq L} S_t(k+n)^{1/2}S_t(k)^{1/2}.
\end{aligned}
\end{equation*}
We use the Cauchy-Schwarz inequality to estimate the first term:
\begin{equation}\label{tailestimate}
\begin{aligned}
\sum_{k\geq M} S_t(k+n)^{1/2}S_t(k)^{1/2}&\leq \left(\sum_{k\geq M} S_t(k+n)\right)^{1/2}\left(\sum_{k\geq M} S_t(k)\right)^{1/2}\leq\\
&\leq \sum_{k=M}^\infty S_t(k)=w_+^2-w_t^2(M).
\end{aligned}
\end{equation}
The second term is only present in the annulus case and can be estimated in an analogous way.   By {\it Condition} 2 again, the difference $w_+^2-w_t^2(M)$ is small for large $M$, uniformly in $t$ on the intervals $t\geq\varepsilon>0$, and so, for $t>0$, $\|x_t\|$ is (locally) the uniform limit of continuous functions and hence continuous. Therefore $\Lambda$ generates a continuous field of Hilbert spaces $(I,\mathcal{H},\Gamma)$.
\end{proof}

Our next concern is with the parametrices $Q_t(k)$. To verify that they form a continuous family of bounded operators in $(I,\mathcal{H},\Gamma)$ we need to check that they are uniformly bounded and that $Q$ maps $\Gamma$ into itself.
We start with the former assertion.

\begin{prop}\label{uniform_est_q_disk}
The norm of $Q_t$ is uniformly bounded in $t$.
\end{prop}

\begin{proof}
First we write $Q_tx_t(k)$ in a more compact form:
\begin{equation*}
Q_tx_t(k) = -\sum_{n=0}^N U^{n} T_t^{(1,n)}f_{n+1}(w_t(k)^2) + \sum_{n=1}^N T_t^{(2,n)}g_{n-1}(w_t(k)^2)\left(U^*\right)^{n}
\end{equation*}
where
\begin{equation*}
\begin{aligned}
T_t^{(1,n)}f(k) &= \sum_{i\geq k}\frac{w_t(k+1)\cdots w_t(k+n)}{w_t(i+1)\cdots w_t(i+n)}\cdot \frac{S_t(i)^{1/2}S_t(i+n+1)^{1/2}}{w_t(k+n)}f(i) \\
T_t^{(2,n)}g(k) &= \sum_{i\leq k}\frac{w_t(i)\cdots w_t(i+n-1)}{w_t(k)\cdots w_t(k+n-1)}\cdot\frac{S_t(i)^{1/2}S_t(i+n-1)^{1/2}}{w_t(i+n-1)}g(i).
\end{aligned}
\end{equation*}
Here the operators $T_t^{(1,n)}$ and $T_t^{(2,n)}$ are integral operators between  weighted $l^2$ spaces, namely:
$T_t^{(1,n)}:l^2_{n+1}\mapsto l^2_n$ and $T_t^{(2,n)}:l^2_{n-1}\mapsto l^2_n$ where
\begin{equation*}
l^2_n:=\{f:\ \ \sum_{k\in\mathbb{S}}S_t(k+n)^{1/2}S_t(k)^{1/2}|f(k)|^2<\infty\}
\end{equation*}

The main technique used to estimate the norms will be the Schur-Young inequality:
if $T: L^2(Y) \longrightarrow L^2(X)$ is an integral operator
$Tf(x) = \int K(x,y)f(y)dy$, then one has

\begin{equation*}
\|T\|^2 \le \left(\underset{x\in X}{\textrm{sup}}\int_Y |K(x,y)|dy\right)\left(\underset{y\in Y}{\textrm{sup}}\int_X |K(x,y)|dx\right).
\end{equation*}
The details can be found in \cite{HS}. 

We will also use two integral estimates, with $t$ independent right hand sides:
\begin{equation}\label{intineq1}
\sum_{i<k}\frac{S_t(k)}{w_t(k)}  \le \int_{w_t(i)^2}^{(w_+)^2} \frac{dx}{\sqrt{x}}=2(w_+-w_t(i))\leq 2(w_+-w_-),
\end{equation}

\begin{equation}\label{intineq2}
\sum_{k\le i}\frac{S_t(k)}{w_t(k)} \le \int_{(w_-)^2}^{w_t(i)^2} \frac{dx}{\sqrt{x}}=2(w_t(i)-w_-)\leq 2(w_+-w_-).
\end{equation}
Such estimates were described and used in \cite{KM1} and are simply obtained by estimating the area under the graph of $x^{-1/2}$, like in the integral test for series.   

First we estimate the norm of $T_t^{(1,n)}$. Repeatedly using the monotonicity of $w_t(i)$ and the Cauchy-Schwarz inequality, we have, like in \cite{KM1}:
\begin{equation}\label{testimate}
\begin{aligned}
&\|T_t^{(1,n)}\|^2 \le \left({\underset{k\in\mathbb{S}}{\textrm{sup }}} \sum_{i\geq k}\frac{S_t(i)^{1/2}S_t(i+n+1)^{1/2}}{w_t(i+n)}\right)\left({\underset{i\in\mathbb{S}}{\textrm{sup }}}\sum_{k\leq i}\frac{S_t(k)^{1/2}S_t(k+n)^{1/2}}{w_t(i+n)}\right) \\
&\le \left[{\underset{k\in\mathbb{S}}{\textrm{sup }}}\left(\sum_{i\geq k} \frac{S_t(i)}{w_t(i)}\right)\left(\sum_{i\geq k}\frac{S_t(i+n+1)}{w_t(i+n)}\right)\right]^{1/2}\left[{\underset{i\in\mathbb{S}}{\textrm{sup }}}\left(\sum_{k\leq i} \frac{S_t(k)}{w_t(k)}\right)\left(\sum_{k\leq i} \frac{S_t(k+n)}{w_t(k+n)}\right)\right]^{1/2}.
\end{aligned}
\end{equation}
Using \ref{wconstineq}, \ref{intineq1}, and \ref{intineq2} we see that the norm of $T_t^{(1,n)}$ is bounded uniformly in $n$ and $t$.
The estimate on $T_t^{(2,n)}$ is essentially the same. Therefore one has

\begin{equation*}
\|Q_t\| \le \ssup{n}{\N} \|T_t^{(1,n)}\| + \ssup{n}{\N} \|T_t^{(2,n)}\| \le \textrm{const}
\end{equation*}
and this completes the proof.
\end{proof}

Next we need to prove that $Q$ maps $\Gamma$ into itself. This requires checking condition (3)
of Proposition \ref{morphprop}. Thus we need to show that, given $x\in\Lambda$, $Qx$ is approximable by  $\Lambda$ at every $t\in I$. The hardest part is to show that this is true around $t=0$, which we will do now.

Let $x\in\Lambda$ be given by formulas \ref{lambdadef} and \ref{lambdadefzero}, and define 
\begin{equation*}
\widetilde{g_{n}}(r^2) := \int_{(w_-)^2}^{r^2}g_{n-1}(\rho^2)\frac{\rho^{n-1}}{r^n}d(\rho^2),
\end{equation*}
and similarly 
\begin{equation*}
\widetilde{f_{n}}(r^2) :=\int_{r^2}^{(w_+)^2}f_{n+1}(\rho^2)\frac{r^{n-1}}{\rho^n}d(\rho^2),
\end{equation*}   
and set
\begin{equation*}
y_t(k) := \sum_{n\le N}U^n\widetilde{f_n}\left(w_t(k)^2\right) + \sum_{n\le N}\widetilde{g_n}\left(w_t(k)^2\right)\left(U^*\right)^n,
\end{equation*}
and for $t=0$: 
\begin{equation*}
y_0(r,\f) := \sum_{n\le N}\widetilde{f_n}(r^2)e^{in\f} + \sum_{n\le N}\widetilde{g_n}(r^2)e^{-in\f}.
\end{equation*}  
Notice that one has $y\in \Lambda$ since clearly $\widetilde{f_{n}}(r^2)$, $\widetilde{g_{n}}(r^2)$ are in $C([(w_-)^2,(w_+)^2])$, and also we have obvious $Q_0x_0=y_0$ which was the motivating property of the above construction of $y$. We will show that $x\in\Lambda$ is approximable by $y\in \Lambda$ at $t=0$. This is stronger than proving that $x$ is approximable by $\Lambda$ at $t=0$.

\begin{prop}\label{nom_conver_oper}
With the above notation the following is true:
\begin{equation*}
\lim_{t\to 0^+}\left\|Q_tx_t-y_t\right\|=0.
\end{equation*}
\end{prop}

\begin{proof}
We show the details for a single $g_n$ term in the finite sum. We first obtain a pointwise estimate. Adding and subtracting we get:
\begin{equation*}
\begin{aligned}
&\left|\sum_{i\leq k}\frac{w_t(i)\cdots w_t(i+n-1)}{w_t(k)\cdots w_t(k+n-1)}\frac{S_t(i)^{1/2}S_t(i+n-1)^{1/2}}{w_t(i+n-1)}g_{n-1}\left(w_t(i)^2\right)-\widetilde{g_{n}}\left(w_t(k)^2\right)\right|\leq\\
&\leq \sum_{i\leq k}\left|\frac{w_t(i)\cdots w_t(i+n-2)}{w_t(k)\cdots w_t(k+n-1)}-\frac{w_t(i)^{n-1}}{w_t(k)^n} \right|S_t(i)\left|g_{n-1}\left(w_t(i)^2\right)\right|+\\
&+\sum_{i\leq k}\frac{w_t(i)\cdots w_t(i+n-2)}{w_t(k)\cdots w_t(k+n-1)}\left|S_t(i)^{1/2}S_t(i+n-1)^{1/2}-S_t(i)\right|\left|g_{n-1}\left(w_t(i)^2\right)\right|+\\
&+\left|\sum_{i\leq k}\frac{w_t(i)^{n-1}}{w_t(k)^n}g_{n-1}\left(w_t(i)^2\right)S_t(i)-\int_{(w_-)^2}^{w_t(k)^2}\frac{\rho^{n-1}}{w_t(k)^n}g_{n-1}(\rho^2)\,d(\rho^2)\right|:=I+II+III.
\end{aligned}
\end{equation*}
Let us discuss the structure of the above terms. The expression inside the absolute value in term $I$ unfortunately in general does not go to zero as $t$ goes to zero. To go around it we show that the expression is small for large $k$ which then lets us use the smallness of $S_t(i)$ to get the desired limit. This term is the trickiest to handle. Term $II$ is the most straightforward to estimate along the lines of the proof of Proposition \ref{normconvergenceuni}.
Finally expression $III$ is a difference between an integral and its Riemann sum, but because of the small denominator it has to be estimated carefully.

To handle term $I$ we need the following observation.

\begin{lem}\label{wratiolemma}
With the above notation we have:
\begin{equation*}
\left|1-\frac{w_t(k)^{n-1}}{w_t(k+1)\cdots w_t(k+n-1)}\right|\leq \sum_{j=0}^{n-1}jh_3(k+n-j),
\end{equation*}
where $h_3(k)$ is the sequence of {\it Condition} 5.
\end{lem}
\begin{proof}
To prove the statement we write
\begin{equation*}
\frac{w_t(k)^{n-1}}{w_t(k+1)\cdots w_t(k+n-1)}=\frac{w_t(k)}{w_t(k+1)}\frac{w_t(k)w_t(k+1)}{w_t(k+1)w_t(k+2)}\frac{w_t(k)\cdots w_t(k+n-2)}{w_t(k+1)\cdots w_t(k+n-1)}
\end{equation*}
and use an elementary inequality:
\begin{equation*}
|1-x_1\cdots x_n|\leq |1-x_1|+\ldots +|1-x_n|
\end{equation*}
if $|x_k|\leq 1$.
\end{proof}
We concentrate on the expression inside the absolute value in term $I$:
\begin{equation*}
\begin{aligned}
&J:=\left|\frac{w_t(i)\cdots w_t(i+n-2)}{w_t(k)\cdots w_t(k+n-1)}-\frac{w_t(i)^{n-1}}{w_t(k)^n} \right|\leq\\
& \leq\left|\frac{w_t(i)\cdots w_t(i+n-2)}{w_t(k)\cdots w_t(k+n-1)}-\frac{w_t(i)^{n-1}}{w_t(k)\cdots w_t(k+n-1)} \right|+\\
&+\left|\frac{w_t(i)^{n-1}}{w_t(k)\cdots w_t(k+n-1)}-\frac{w_t(i)^{n-1}}{w_t(k)^n} \right|.
\end{aligned}
\end{equation*}
Factoring we get:
\begin{equation*}
\begin{aligned}
&J\leq\frac{1}{w_t(k+n-1)}\left|1-\frac{w_t(i)^{n-2}}{w_t(i+1)\cdots w_t(i+n-2)}\right|+\\
&+\frac{1}{w_t(k)}\left|1-\frac{w_t(k)^{n-1}}{w_t(k+1)\cdots w_t(k+n-1)}\right|.
\end{aligned}
\end{equation*}
Using lemma \ref{wratiolemma} yields:
\begin{equation*}
\begin{aligned}
&J\leq\frac{1}{w_t(k+n-1)}\sum_{j=0}^{n-2}jh_3(i+n-1-j)+\frac{1}{w_t(i)}\sum_{j=0}^{n-1}jh_3(k+n-j)=\\
&=:\frac{1}{w_t(k+n-1)}h_4(i)+\frac{1}{w_t(i)}h_5(k).
\end{aligned}
\end{equation*}
The functions $h_4(k)$ and $h_5(k)$ above are $t$ independent and go to zero as $k\to\pm\infty$. Consequently:
\begin{equation*}
\begin{aligned}
I(k)&\leq \textrm{const}\,\frac{1}{w_t(k+n-1)}\sum_{i\leq k}S_t(i)h_4(i)+\textrm{const}\,h_5(k)\sum_{i\leq k}\frac{S_t(i)}{w_t(i)}\leq\\
&\leq\textrm{const}\,\frac{1}{w_t(k+n-1)}\sum_{i\leq k}S_t(i)h_4(i)+\textrm{const}\,h_5(k)=:I_1+I_2.
\end{aligned}
\end{equation*}
We use the following lemma to handle both $I_1$ and $I_2$. This is the tricky part of the argument.

\begin{lem}\label{transferlemma}
If $h(k)\to 0$ as $k\to\pm\infty$ then
\begin{equation*}
\lim_{t\to 0^+}\sum_{k\in\mathbb{S}} S_t(k)h(k)=0.
\end{equation*}
\end{lem}
\begin{proof}
We split the sum:
\begin{equation*}
\begin{aligned}
&\sum_{k\in\mathbb{S}} S_t(k)h(k)=\sum_{|k|\leq N} S_t(k)h(k)+\sum_{|k|>N} S_t(k)h(k)\leq\\
\leq\,&\textrm{const}\,\sum_{|k|\leq N} S_t(k) +\textrm{const}\,{\underset{|k|>N}{\textrm{sup }}}h(k)
\end{aligned}
\end{equation*}
and first choose $N$ such that ${\underset{|k|>N}{\textrm{sup }}}h(k)\leq \varepsilon/2$ and then choose $\delta>0$ such that
$\sum\limits_{|k|\leq N} S_t(k)\leq \varepsilon/2$ far all $t\leq\delta$. The last inequality is possible because of \it{Condition} 3.
\end{proof}

As a corollary we also have:
\begin{equation}\label{transferlimit}
\lim_{t\to 0^+}\sum_{k\in\mathbb{S}} S_t(k+n)^{1/2}S_t(k)^{1/2}h(k)=0,
\end{equation}
obtained by estimating like in \ref{tailestimate}:
\begin{equation*}
\begin{aligned}
&\sum_{k\in\mathbb{S}} S_t(k+n)^{1/2}S_t(k)^{1/2}h(k)\leq\\
\leq &\left(\sum_{k\in\mathbb{S}} S_t(k+n)\right)^{1/2}\left(\sum_{k\in\mathbb{S}} S_t(k)h(k)^2\right)^{1/2}\leq \textrm{const}\,\left(\sum_{k\in\mathbb{S}} S_t(k)h(k)^2\right)^{1/2}.
\end{aligned}
\end{equation*}

We proceed to show that the norms $I_1$ and $I_2$ are small for small $t$. This is more straightforward with the $I_2$ term. Namely we have $\left\|I_2\right\|^2\leq \textrm{const}\,\sum_{k\in\mathbb{S}} S_t(k+n)^{1/2}S_t(k)^{1/2}h_5^2(k)$
which by \ref{transferlimit} goes to zero as $t$ goes to zero.

To estimate $I_1$ we notice first that
\begin{equation*}
I_1(k)\leq \textrm{const}\,\sum_{i\leq k}\frac{S_t(i)}{w_t(i)}h_4(i)\leq\textrm{const}\,\sum_{i\leq k}\frac{S_t(i)}{w_t(i)}\leq\textrm{const}
\end{equation*}
by \ref{intineq2}. Consequently we have:
\begin{equation*}
\begin{aligned}
\left\|I_1\right\|^2&=\sum_{k\in\mathbb{S}} S_t(k+n)^{1/2}S_t(k)^{1/2}I_1^2(k)\leq \textrm{const}\,\sum_{k\in\mathbb{S}}S_t(k+n)^{1/2}S_t(k)^{1/2}I_1(k)\leq\\
&\leq\textrm{const}\,\sum_{i,k\in\mathbb{S}}\frac{S_t(k+n)^{1/2}S_t(k)^{1/2}}{w_t(k+n)}S_t(i)h_4(i)\leq\textrm{const}\,\sum_{i\in\mathbb{S}} S_t(i)h_4(i).
\end{aligned}
\end{equation*}
The sum over $k$ above is estimated as in \ref{testimate}, and we can use Lemma \ref{transferlemma} again to conclude that $\left\|I_1\right\|^2$ goes to zero as $t$ goes to zero.

Now we estimate term $II$. This is done analogously to the way we treated the second term in Proposition \ref{normconvergenceuni}. Using the boundedness of $g_{n-1}$, the definition of $h_2(t)$, and \ref{intineq2}, we have:
\begin{equation*}
\begin{aligned}
II(k)&\leq \sum_{i\leq k} \frac{\left|S_t(i+n-1)^{1/2}S_t(i)^{1/2}-S_t(i)\right|}{w_t(i+n-1)}\left|g_{n-1}(w_t(i)^2)\right|\leq\\
&\leq\textrm{const}\sum_{i\leq k} \frac{S_t(i)}{w_t(i)}\,\left|\left(\frac{S_t(i+n-1)}{S_t(i)}\right)^{1/2}-1\right|\leq \textrm{const}\,h_2(t).
\end{aligned}
\end{equation*}
Consequently $\left\|II\right\|^2\leq \textrm{const}\,h^2_2(t)$
which goes to zero by {\it Condition} 4.

Finally we estimate $III(k)$. It is clear that this expression is small for small $t$ and a fixed $k$, as a difference between an integral and its Riemann sum. However this is not enough in the disk case when $w_t(k)^n$ in the denominator is small for small $t$. To overcome this difficulty we first replace $g_{n-1}$ by its step function approximation and then deal directly with the remaining integral of $\rho^{n-1}$.

With this strategy in mind we estimate:
\begin{equation*}
\begin{aligned}
&III(k)=\left|\sum_{i\leq k}\frac{w_t(i)^{n-1}}{w_t(k)^n}g_{n-1}\left(w_t(i)^2\right)S_t(i)-\int_{(w_-)^2}^{w_t(k)^2}\frac{\rho^{n-1}}{w_t(k)^n}g_{n-1}(\rho^2)\,d(\rho^2)\right|\leq\\
&\leq \left|\sum_{i\leq k}\left(\frac{w_t(i)^{n-1}}{w_t(k)^n}g_{n-1}\left(w_t(i)^2\right)S_t(i)-\int_{w_t(i-1)^2}^{w_t(i)^2}\frac{\rho^{n-1}}{w_t(k)^n}g_{n-1}\left(w_t(i)^2\right)\,d(\rho^2)\right)\right|+\\
&+\left|\sum_{i\leq k}\int_{w_t(i-1)^2}^{w_t(i)^2}\frac{\rho^{n-1}}{w_t(k)^n}\left(g_{n-1}\left(w_t(i)^2\right)-g_{n-1}(\rho^2)\right)\,d(\rho^2)\right|
=:III_1(k)+III_2(k).
\end{aligned}
\end{equation*}
Since continuous functions on a closed interval are uniformly continuous, the function
\begin{equation*}
h_6(t):=\ssup{i}{\mathbb{S}}\underset{\rho^2\in[{(w_t(i-1)^2},{(w_t(i))^2}]}{\textrm{ sup }}\left|g_{n-1}\left(w_t(i)^2\right)-g_{n-1}\left(\rho^2\right)\right|
\end{equation*}
goes to zero as $t\to0^+$. Consequently, using the definition of $h_6(t)$, term $III_2$ can be estimated as follows:
\begin{equation*}
III_2(k)\leq h_6(t)\int_{(w_-)^2}^{w_t(k)^2}\frac{\rho^{n-1}}{w_t(k)^n}\,d(\rho^2)\leq h_6(t)w_t(k)\int_0^1u^{n-1}\,d(u^2)\leq\textrm{const}\,h_6(t).
\end{equation*}
This means that $\|III_2\|$ goes to zero as $t\to0^+$.

When estimating $III_1$ we first eliminate $g_{n-1}$ using its boundedness: 
\begin{equation*}
\begin{aligned}
&III_1(k)=\left|\sum_{i\leq k}\int_{w_t(i-1)^2}^{w_t(i)^2}\left(\frac{w_t(i)^{n-1}}{w_t(k)^n}g_{n-1}\left(w_t(i)^2\right)-\frac{\rho^{n-1}}{w_t(k)^n}g_{n-1}\left(w_t(i)^2\right)\right)d(\rho^2)\right|\leq\\
&\leq\textrm{const}\sum_{i\leq k}\int_{w_t(i-1)^2}^{w_t(i)^2}\left(\frac{w_t(i)^{n-1}}{w_t(k)^n}-\frac{\rho^{n-1}}{w_t(k)^n}\right)d(\rho^2).
\end{aligned}
\end{equation*}
What is left is the difference between the integral of $\rho^{n-1}$ and its upper sum which we handle like in the error estimate of the integral test for series. This is summarized in the following sequence of inequalities.
\begin{equation*}
\begin{aligned}
&III_1(k)\leq \textrm{const}\sum_{i\leq k}\left(\frac{w_t(i)^{n-1}}{w_t(k)^n}-\frac{w_t(i-1)^{n-1}}{w_t(k)^n}\right)S_t(i)\leq\\
&\leq \textrm{const}\left(\sum_{i\leq k}\frac{w_t(i)^{n-1}}{w_t(k)^n}S_t(i)-\sum_{i\leq {k-1}}\frac{w_t(i)^{n-1}}{w_t(k)^n}S_t(i+1)\right)\leq\\
&\leq \textrm{const}\sum_{i\leq k-1}\frac{w_t(i)^{n-1}}{w_t(k)^n}S_t(i)\left(1-\frac{S_t(i+1)}{S_t(i)}\right)+\textrm{const}\,\frac{S_t(k)}{w_t(k)}.
\end{aligned}
\end{equation*}
Notice that
\begin{equation*}
\frac{S_t(k)^2}{w_t(k)^2}=S_t(k)\,\frac{w_t(k)^2-w_t(k-1)^2}{w_t(k)^2}\leq S_t(k)\leq h_1(t).
\end{equation*}
Hence, using the monotonicity of $w_t(i)$ we have
\begin{equation*}
III_1(k)\leq \textrm{const}\,h_2(t)\sum_{i\leq k-1}\frac{w_t(i)^{n}}{w_t(k)^n}\frac{S_t(i)}{w_t(i)}+\textrm{const}\sqrt{h_1(t)}\leq
\textrm{const}\left(h_2(t)+\sqrt{h_1(t)}\right),
\end{equation*}
and again $\| III_1\|$ goes to zero as $t\to0^+$.
The proof of the proposition is complete. \end{proof}

To proceed further we need a better understanding of $\Gamma$, the space of continuous sections of our  continuous field. We have the following useful result.

\begin{lem}\label{uniform_approx}
For $t>0$ consider the following expression

\begin{equation*}
x_t(k) = \sum_{n\leq N} U^n F_n(t,k) + \sum_{n \leq N} G_n(t,k)(U^*)^n
\end{equation*}
such that the functions $t\mapsto F_n(t,k)$ and $t\mapsto G_n(t,k)$ are continuous for every $k$, and such that $|F_n(t,k)|$ and $|G_n(t,k)|$ are bounded (in both variables). Then $x_t$ is approximable by $\Lambda$ at every $t>0$.
\end{lem}

\begin{proof}
Without a loss of generality we may assume that $x_t(k)=U^nF_n(t,k)$  as the proof is identical for the $G$ terms, and it will extend to finite sums of such $x$'s. Given $t_0\in I$ and $\epsilon>0$, let $y\in\Lambda$ be such that for $t>0$
\begin{equation*}
y_t(k):=U^nf_n(w_t(k)^2),
\end{equation*}
where we choose $f_n\in C([(w_-)^2,(w_+)^2])$ such that $\left\|F_n(t_0,\cdot) - f_n\left(w_{t_0}(\cdot)^2\right)\right\| \leq \frac{\varepsilon}{2}$.
This is always possible since the space of sequences of the form $k\to f_n(w_{t_0}(k)^2)$, where $f_n\in C([(w_-)^2,(w_+)^2])$, is a dense subspace in the Hilbert space $l^2_n$.

We want to show that 
\begin{equation*}
\left\|x_t - y_t\right\| \leq \varepsilon
\end{equation*}
for all $t$ sufficiently close to $t_0$. By the construction of $f_n$ this is true at $t=t_0$. We will prove that $t\to \left\|x_t - y_t\right\|$ is continuous for $t>0$ which will imply the above inequality. But the inequality means that $x$ is approximable by $\Lambda$ at $t=t_0$, which is exactly what we want to achieve.

The proof that $t\to \left\|x_t - y_t\right\|$ is continuous is analogous to the last part of the proof of Theorem \ref{cont_hil_sp}, that the norm is continuous for elements of $\Lambda$ and $t>0$. Indeed, by the continuity assumptions, $\left\|x_t - y_t\right\|^2$ is an infinite sum of continuous functions:
\begin{equation*}
\left\|x_t - y_t\right\|^2=\sum_{k\in\mathbb{S}} S_t(k+n)^{1/2}S_t(k)^{1/2}\left|F_n(t,k)-f_n(w_t(k)^2)\right|^2.
\end{equation*}

The series converges uniformly around $t_0$ because, by the boundedness assumptions, we can estimate the remainder as follows:
\begin{equation*}
\sum_{k\geq M} S_t(k+n)^{1/2}S_t(k)^{1/2}\left|F_n(t,k)-f_n(w_t(k)^2)\right|^2\leq\textrm{const}\,\sum_{k\geq M} S_t(k+n)^{1/2}S_t(k)^{1/2}.
\end{equation*}
For large $M$ this is small by \ref{tailestimate}. In the annulus case there is also a remainder at $-\infty$ which also goes to zero by an analogous estimate. As a consequence $t\to \left\|x_t - y_t\right\|$ is indeed continuous for $t>0$ and the lemma is proved.
\end{proof}

We now have all the tools to finish the proof the second theorem.

\begin{proof}(of Theorem \ref{cont_fam_oper})

What remains is to show that $Q_tx_t$ is approximable by $\Lambda$ for $t>0$ since
Propositions \ref{uniform_est_q_disk} and \ref{nom_conver_oper} establish the other properties of $\{Q_t\}$ needed to conclude that they form a continuous family of bounded operators in $(I, \mathcal{H}, \Gamma)$.  

To prove that $Q_tx_t$ is approximable by $\Lambda$ for $t>0$ we use Lemma \ref{uniform_approx} with
\begin{equation*}
\begin{aligned}
F_n(t,k) &= \sum_{i\geq k}\mathcal{F}_n(t,i) := \sum_{i\geq k}\frac{w_t(k+1)\cdots w_t(k+n)}{w_t(i+1)\cdots w_t(i+n)}\cdot \frac{S_t(i)^{1/2}S_t(i+n+1)^{1/2}}{w_t(k+n)}f_{n+1}(i) \\
G_n(t,k) &= \sum_{i\leq k} \mathcal{G}_n(t,i) := \sum_{i\leq k}\frac{w_t(i)\cdots w_t(i+n-1)}{w_t(k)\cdots w_t(k+n-1)}\cdot\frac{S_t(i)^{1/2}S_t(i+n-1)^{1/2}}{w_t(i+n-1)}g_{n-1}(i).
\end{aligned}
\end{equation*}
Thus we need to show that $F_n(t,k)$ and $G_n(t,k)$ are continuous and bounded functions of $t$, for $t>0$.    This will be done for the $F_n(t,k)$ term only as the argument is analogous for the $G_n(t,k)$ term.  In fact, in the disk case the $G_n(t,k)$ is only a finite sum, so the continuity for $t>0$ follows immediately from {\it Condition} 2.  

Each $\mathcal{F}_n(t,i)$ is continuous on the intervals $t\ge\varepsilon>0$ by {\it Condition} 2, so we must show that for each $k$, the series defining $F_n(t,k)$ converges uniformly in $t$.   To estimate the tail end of the series we use \ref{wconstineq}, \ref{tailestimate}, and the boundedness of $f_{n+1}(i)$ to get

\begin{equation*}
\begin{aligned}
&\left|\sum_{i=M}^\infty \frac{w_t(k+1)\cdots w_t(k+n)}{w_t(i+1)\cdots w_t(i+n)}\cdot\frac{S_t(i)^{1/2}S_t(i+n+1)^{1/2}}{w_t(k+n)}f_{n+1}(i)\right|\leq \\
&\leq \frac{\textrm{const}}{w_t(k+n)}\sum_{i=M}^\infty S_t(i)^{1/2}S_t(i+n+1)^{1/2}\leq \frac{\textrm{const}}{w_t(k+n)}\left(w_+^2 - w_t^2(M)\right),
\end{aligned}
\end{equation*}
which goes to zero uniformly on the intervals $t\ge\varepsilon>0$ as $M$ goes to infinity by {\it Condition} 2. Hence $F_n(t,k)$ is a uniform limit of continuous functions and consequently it is continuous for $t>0$ and for each $k$.   

Next we show that $F_n(t,k)$ and $G_n(t,k)$ are bounded. We have:

\begin{equation*}
\begin{aligned}
|F_n(t,k)| &\le \textrm{const} \sum_{i\geq k}\frac{S_t(i)^{1/2}S_t(i+n+1)^{1/2}}{w_t(i+n)}\le \\
&\le  \textrm{const}\left(\sum_{i\geq k} \frac{S_t(i)}{w_t(i)}\right)^{1/2}\left(\sum_{i\geq k}\frac{S_t(i+n+1)}{w_t(i+n)}\right)^{1/2}\le  \textrm{const},
\end{aligned}
\end{equation*}
where we used  \ref{wconstineq}, \ref{intineq1}, and \ref{intineq2}. Similar argument works also for estimating $|G_n(t,k)|$. Thus the assumptions of Lemma \ref{uniform_approx} are satisfied and $Q_tx_t$ is approximable by $\Lambda$ at every $t$.
Hence the collection $\{Q_t\}$ is a continuous family of bounded operators.   This finishes the proof.
\end{proof}

\end{document}